\documentclass{amsart}
\usepackage{amsmath}
\usepackage{amsfonts}
\usepackage{amstext}
\usepackage{amsbsy}
\usepackage{amsopn}
\usepackage{amsxtra}
\usepackage{upref}
\usepackage{amsthm}
\usepackage{amsmath}
\usepackage{amssymb}

\newtheorem{prop}{Proposition}[section]
\newtheorem{rem}{Remark}[section]
\newtheorem{lema}{Lemma}[section]
\newtheorem{defi}{Definition}[section]
\newtheorem{teo}{Theorem}[section]

\newtheorem*{claim*}{Claim}

\newtheorem{maintheorem}{Theorem}

\def\phi{\varphi}

\def\N{{\mathbb N}}

\DeclareMathOperator{\card}{card}

\title{Hausdorff Dimension of Cantor Series}
\date{\today}

\begin{thanks} { G. I. would like to acknowledge the support of
    Proyecto Fondecyt 11070050.  B. S. would like to acknowledge the
    support of Proyecto Fondecyt 11060538.  The authors would also like to
    thank Mike Todd for useful conversations on this subject.}
\end{thanks}

\author{Godofredo Iommi} \address{Facultad de Matem\'aticas,
Pontificia Universidad Cat\'olica de Chile (PUC), Avenida Vicu\~na Mackenna 4860, Santiago, Chile}
\email{giommi@mat.puc.cl}
\urladdr{http://www.mat.puc.cl/\textasciitilde giommi/}
\author{Bart{\l}omiej Skorulski} \address{Universidad cat\'olica del Norte, Facultad de Ciencias, Departamento de Matem\'aticas, Avenida Angamos 0610, Antofagasta, Chile}
\email{bskorulski@ucn.cl}


\begin{document}

\begin{abstract}
  In 1996 Y. Kifer obtained a variational formula for the Hausdorff
  dimension of the set of points for which the frequencies of the
  digits in the Cantor series expansion is given. In this note we
  present a slightly different approach to this problem that allow us
  to solve the variational problem of Kifer's formula.
\end{abstract}

\maketitle
\section{Introduction}
Let $A=\{b_n\}_{n \in \N}$ be a sequence of positive integers such that $b_n
\in \N \setminus \{1\}$. Every real number $x \in [0,1]$ can be
written as
\begin{equation*}
  x = \sum_{n=1}^{\infty} \frac{\epsilon_n(x)}{b_1 b_2\cdots b_n},
\end{equation*}
where $\epsilon_n(x) \in \{0, 1 , \dots , b_n -1 \}$. We write
\begin{equation*}
  x= \left[\epsilon_1(x) \epsilon_2(x) \dots \epsilon_n(x) \dots \right]_A
\end{equation*}
and call it the $A-$\emph{Cantor Series} of $x$ with respect to the
base $A=\{b_n \}_{n\in \N}$. For every base $A$ the Cantor series is unique
except for a countable number of points. Note that if $A$ is the sequence such that for every $n \in \N$ we have $b_n= b$, then the $A-$Cantor series corresponds to the base $b$ expansion of $x \in [0,1]$.

In 1996 Kifer (see \cite{kif}) studied the problem of  computing the size of the level sets
determined by the frequency of digits in the $A$-Cantor series
expansion. More precisely, for each $n \in \mathbb{N}$, $k \in
\mathbb{N}$ and $\textrm{ } x \in (0,1)$ set
\begin{equation*}
\tau_{j}(x,n):= \card\{ 1\leq k \leq n : \epsilon_{k}(x)=j\}.
\end{equation*}
Whenever there exists the limit
\begin{equation} \label{deffre}
  \tau_{j}(x)= \lim_{n \rightarrow \infty } \frac{\tau_{j}(x,n)}{n},
\end{equation}
it is called the \emph{frequency} of the number $j$ in the $A-$Cantor
series expansion of $x$. A sequence $\alpha= \{\alpha_n \}_{n \in \N}$ is called a
\emph{stochastic vector}, if $\sum_{j=0}^{\infty}
\alpha_j = 1$ and for every $j \in \N$ we have $\alpha_j \geq 0$. We consider  the set
of points for which the frequency of the digit $j$ is equal to
$\alpha_j$, that is
\begin{equation}\label{eq:5}
J_{A}(\alpha)= \left\{x \in  [0,1] : \tau_j(x)= \alpha_j \text { for every }
j \in \{0, 1 , \dots\}       \right\}.
\end{equation} 
The question we are interested in is: \emph{What is the size of these sets?}

In the case that the sequence $A$ is constant, with $b_n=b$, this
problem was studied in 1949 by Eggleston (see \cite{eg}). In this setting the
only digits that appear (hence, that can have positive frequency) are
$\{0,1 , \dots, b-1\}$. Note that, by Borel Normal Number Theorem, the
Lebesgue measure of $J(\alpha)$ is positive if, and only if,
$\alpha_j=1/b$ for all $j=0,\ldots, b-1$. Therefore, in order to
quantify the size of $J(\alpha)$  Eggleston considered its Hausdorff
dimension, that we denote by $\dim_H(\cdot)$. In \cite{eg} he proved that
\begin{equation*}
  \dim_H(J(\alpha)) = \frac{\sum_{j=0}^{b-1} \alpha_j \log \alpha_j}{\log b}.
\end{equation*}

The general problem, when the sequence $A$ is non-constant, is more
subtle. Indeed, the frequency of the digit $k$ depends also on the
frequency of the base $b_n$ in the sequence $A$. For example, let  
$A=(2,2,3,2,2,3,2,2,3, \dots)$. That is, the frequency of the
base $3$ in the sequence $A$ is equal to $1/3$ and of the base $2$ is
equal to $2/3$. Since the digit $k=2$ can only appear when $b_n=3$ we
have that for every $x \in [0,1]$ the following bound holds $\tau_2(x)
\leq 1/3$.

The first studies of Hausdorff dimension of sets of numbers
defined in terms of their frequencies of digits in a Cantor expansion
were done by Preyri\`ere \cite{pre}. He considered sets for
which the conditional frequencies of the digit $i$ subject to the
frequency of the base $b_n$ are fixed and computed its Hausdorff
dimension. Let us be more precise. Define the frequency of the base
symbol $k$ in the sequence $A$ by
\begin{equation} \label{da} 
  d_{k}= \lim_{n \to \infty} (1/n)D_k(n),
\end{equation} 
where
\begin{equation} \label{D-k}
  D_k(n) =\card\{i \in \{1,2, \dots, n \} : b_i=k \}.
\end{equation}
Let
\begin{equation} \label{P}
  \pi=\Big\{ P=(p_{n,j})_{
      \begin{subarray}{l}
        j=0,\ldots,n-1\\
        n=2,\ldots
      \end{subarray}} : \sum_{j=0}^{b_n-1} p_{n,j} =1\textrm{, } 
  p_{n,j}\geq 0\Big\}.
\end{equation}
Denote by 
\begin{equation} \label{tau-k-m}
  \tau_{k,j}(x,n):=\card\{i \in  \{1,2, \dots, n\}:
  b_i=k \text{ and } \epsilon_i(x)=j     \},
\end{equation}
and define the frequency by
\begin{equation}
\tau_{k,m}(x):= \lim_{n \to \infty}  \tau_{k,m}(n,x)
\end{equation}
whenever the limit exits. For $P\in\pi$ consider the set
\begin{equation} \label{set}
  J_{A}(P)= \{x \in [0,1] :    
  \tau_{k,b_n}(x)=d_{b_n} p_{k,b_n}  , n \in \N  \}.
\end{equation}
In 1977  Peyri\`ere, see \cite{pre}, proved that:
\begin{teo}[Peyri\`ere]
  Let $A=\{b_n\}_{n \in \N}$ be a sequence of positive integers $b_n
  \in \N \setminus \{1\}$ such that $\sum_{n=2}^{\infty} d_n =1$.  Let
  $P\in \pi$, then
\begin{eqnarray}
  \label{eq:3}
  \dim_H(J_{A}(P))  = 
  \frac{\sum_{n=2}^{\infty} d_{n} \sum_{j=0}^{n-1} p_{n,j} \log p_{n,j}} 
  {\sum_{n=1}^{\infty} d_{n} \log n}.
\end{eqnarray}
\end{teo}
Let us stress that the set $J_{A}(\alpha)$ defined by \eqref{eq:5},
where $\alpha$ is a stochastic vector, has a much more complicated
structure than the sets $J_{A}(P)$ considered in \eqref{set}. Not only the set
$J_A(\alpha)$ contains a non-denumerable family of sets of the form
$J_{A}(P)$. Indeed, if $x \in J_{A}(P)$ then
\begin{equation}
  \tau_{j}(x) =\sum_{n=1}^{\infty} \tau_{n,j}(x)= \sum_{n=1}^{\infty} d_n p_{n,j}.
\end{equation}
But, it is possible for $\tau_j(x)$ to exists whereas $\tau_{n,j}(x)$
does not exists for any $n \in  \N$. 

Kifer obtained a formula relating the Hausdorff dimension of $J_A(\alpha)$ with the Hausdorff dimension of certain sets $J_A(P)$. Let
\begin{displaymath}
  \pi(\alpha):=\{P\in\pi:\alpha_j=\sum_{n=1}^{\infty} d_n p_{n,j}\}.
\end{displaymath}
In \cite{kif}, Kifer proved that
\begin{teo}[Kifer]
Let $A=\{b_n\}_{n \in \N}$ be a sequence of positive integers $b_n \in \N \setminus \{1\}$
such that $\sum_{n=2}^{\infty} d_n =1$. Let $\{\alpha_n\}_{n \in \N}$ be a stochastic vector, then\begin{equation}\label{eq:4}
  \dim_H(J(\alpha)) 
  =\sup_{P\in \pi(\alpha)} \dim_H(J_A(P))=
  \sup_{P\in \pi(\alpha)}\frac{\sum_{n=1}^{\infty} d_{n} \sum_{j=0}^{n-1}
        p_{n,j} \log p_{n,j}} {\sum_{n=1}^{\infty} d_{n} \log n}.
\end{equation}
 \end{teo}
 This formula has however a disadvantage. Even in very simple cases it
 is difficult to calculate the Hausdorff dimension of $J(\alpha)$.  We
 propose a slightly different approach that will allow us to overcome
 this difficulty. This approach is based in the following sequences:

Put $j_0:=\min\{j:\alpha_j\neq 0\}$ and $d_1:=0$.  Let
\begin{displaymath}
  A_n:=\sum_{j=j_0}^{n-1} \alpha_j-\sum_{k=j_0+1}^{n-1}d_k
\end{displaymath}
and
\begin{displaymath}
  r_n:=\frac{1}{A_n}\prod_{k=j_0+1}^{n-1}\Big(1-\frac{d_{k}}{A_{k}}\Big)
  \textrm{ and }t_j:=\frac{\alpha_j}
  {\prod_{k=j_0+1}^{j}\Big(1-\frac{d_{k}}{A_{k}}\Big)}.
\end{displaymath}
We prove the following,

\begin{maintheorem} \label{fre} Let $A=\{b_n\}_{n \in \mathbb{N}}$ be a sequence of positive integers $b_n \in \N \setminus \{1\}$ such that $\sum_{n=2}^{\infty} d_n =1$
and $\alpha=\{\alpha_i\}_{i \in \mathbb{N}}$ be a stochastic vector then 
  \begin{displaymath} \label{m}
  \dim_H(J(\alpha)) 
    =\frac{\sum_{j=j_0}^{\infty} \alpha_j \log t_j + \sum_{i=j_0+1}^{\infty}
      d_{i} \log r_i }{ \sum_{n=2}^{\infty} d_{n} \log n}.
  \end{displaymath}
\end{maintheorem}
It is direct consequence of the above Theorem that the  supremum in Kifer's result \eqref{eq:4} is attained at  the level $J_A(P^\alpha)$, where
\begin{equation}
  \label{eq:375}
  P^\alpha=(p_{n,j}^\alpha)\textrm{ and }p_{n,j}^\alpha=r_n t_j.
\end{equation}
This is yet another example of the phenomenon described by Cajar in
\cite{ca} where the Hausdorff dimension of a non-denumerable union of
sets corresponds to the supremum of the Hausdorff dimension of each
set. This property is, of course, false in general, but in the case of
sets defined in terms of the frequencies of digits holds for a large
class of systems, see \cite{ca}.

\begin{rem}
The problem considered in this note can be addressed using techniques from random dynamical systems, but with those techniques we were not able to obtain better results than the ones presented  here.  The main formulas \eqref{eq:3}, \eqref{eq:4} and \eqref{m} can be thought of as the quotient of a random entropy over a random Lyapunov exponent.
Compare with the work of Kifer  \cite{kif, kif2}.
\end{rem}

\section{Proof of Theorem \ref{fre}}

The following basic result from dimension theory will be used several times in the rest of note, for a detailed exposition on the subject see \cite{fa,pe}.  Let $\mu$ be a Borel
finite measure, the \emph{pointwise dimension} of $\mu$ at the point
$x$ is defined, whenever the limit exists, by
\begin{equation} \label{pointdim}
d_{\mu}(x) = \lim_{r \to 0} \frac{\log \mu (B(x,r))}{\log r},
\end{equation}
where $B(x,r)$ is the ball of center $x$ and radius $r$. The following result can be found in  \cite[p.42]{pe},
\begin{prop} \label{ub} Let $\mu$ be a finite Borel measure. If
  $d_{\mu}(x) \leq d$, for every $x \in F$, then $\dim_H(F) \leq d$.
\end{prop}

\subsection*{  Construction of the optimal sequence}
We start with a general lemma. Let $L\in \N\cup \{\infty\}$ and let
$d=(d_n)_{k=1}^L$ and $\alpha=(\alpha_j)_{j=0}^{L-1}$ be stochastic
vectors. Put $j_0:=\min\{j:\alpha_j\neq 0\}$.
\begin{lema}
\label{lem:general}
  There exist $(r_n)_{n=j_0+1}^{L}$ and $(t_{j_0})_{j=0}^{L-1}$ where $r_n,
  s_j\geq 0$ such that for 
  \begin{equation}
    \label{eq:377}
    p_{n,j}=r_n t_j
  \end{equation}
  we have that
  \begin{enumerate}
  \item\label{item:lema1} $\sum_{j=j_0}^{n-1} p_{n,j}=1$ for every
    $n\geq j_0+1$ and
  \item\label{item:lema2} $\sum_{k=j+1}^{L}d_{k}p_{k,j}=\alpha_j$ for
    every $j\geq j_0$. 
  \end{enumerate}
\end{lema}
\begin{proof}
  Let $n\in \N$ be such that $j_0< n\leq L$, we define inductively the
  numbers $A_n$ and $\alpha_{j}^{(n)}$.  For every $ j \geq j_0$ let
  \begin{displaymath}
    \alpha_j^{(j_0)}:=\alpha_j \textrm{ and } A_{j_0}:=\alpha_{j_0}.
  \end{displaymath}
  If $j\geq n-1$ then we define
  $\alpha_j^{(n)}:=\alpha_j^{(n-1)}=\alpha_j$. If $j_0\leq j < n-1$ then we
  define
  \begin{equation}
    \label{eq:200}
    \alpha_j^{(n)}:=\alpha_j^{(n-1)}\Big(1-\frac{d_{n-1}}{A_{n-1}}\Big)=
    \alpha_j\prod_{k=j+1}^{n-1}\Big(1-\frac{d_{k}}{A_{k}}\Big)
    =\alpha_j\frac{\prod_{k=j_0+1}^{n-1}\Big(1-\frac{d_{k}}{A_{k}}\Big)}
    {\prod_{k=j_0+1}^{j}\Big(1-\frac{d_{k}}{A_{k}}\Big)}
  \end{equation}
  and
  \begin{equation}
    \label{eq:300}
    A_n:=\sum_{j=j_0}^{n-1}\alpha_j^{(n)}.
  \end{equation}
  Put
  \begin{displaymath}
    r_n:=\frac{1}{A_n}\prod_{k=j_0+1}^{n-1}\Big(1-\frac{d_{k}}{A_{k}}\Big)
    \textrm{ and }t_j:=\frac{\alpha_j}{\prod_{k=j_0+1}^{j}\Big(1-\frac{d_{k}}{A_{k}}\Big)}.
  \end{displaymath}
  It follows from equation \eqref{eq:377} and \eqref{eq:200} that
  \begin{equation}
    \label{eq:350}
    p_{n,j}:=\frac{\alpha_j^{(n)}}{A_n}.
  \end{equation}

  Note that from equation \eqref{eq:300} we have
  \begin{displaymath}
    \sum_{j=j_0}^{n-1} p_{n,j}=   
    \sum_{j=j_0}^{n-1} \frac{\alpha_j^{(n)}}{A_n}
    =\frac{1}{A_n}  \sum_{j=j_0}^{n-1} \alpha_j^{(n)}= \frac{A_n}{A_n}=1. 
  \end{displaymath}
  This proves item (\ref{item:lema1}).

Note that applying equation \eqref{eq:350} and \eqref{eq:200} we
obtain,
\begin{equation}
  \label{eq:400}
  \alpha_j^{(n)}=\alpha_j^{(n-1)}-d_{n-1}p_{n-1,j}
  = \alpha_j-\sum_{k=j+1}^{n-1}d_{k}p_{k,j}.
\end{equation}
Also note that if $j\leq n-1$, from equation \eqref{eq:400} and from
 item (\ref{item:lema1}) of the Lemma (that we already proved), we have that
\begin{multline*}
  A_n=\sum_{j=j_0}^{n-1} \alpha_j^{(n)} = \sum_{j=j_0}^{n-1} \alpha_j
  - \sum_{j=j_0}^{n-1}\sum_{k=j+1}^{n-1} d_k p_{k,j} 
  =\sum_{j=0}^{n-1} \alpha_j - \sum_{k=j_0+1}^{n-1}d_k\sum_{j=j_0}^{k-1} 
  p_{k,j}\\
  =\sum_{j=j_0}^{n-1} \alpha_j-\sum_{k=j_0+1}^{n-1}d_k
  = 
  \left( 1-\sum_{k=j_0+1}^{n-1}d_k \right) - 
  \left( 1-  \sum_{j=j_0}^{n-1} \alpha_j \right) 
  =  \sum_{k=n}^{L}d_k-\sum_{j=n}^{L-1} \alpha_j.
\end{multline*}
Therefore, if $L=\infty$, then 
$$ \lim_{n \to \infty} A_n = 0. $$
We have proved that if $L= \infty$ then the series $\lim_{n \to
  \infty} A_n=\sum_{j=j_0}^{\infty} \alpha_j^{(n)}= 0$. In particular,
$\lim_{n \to \infty} \alpha_j^{(n)}=0$.  Then, from equation
\eqref{eq:400} we obtain that
\begin{equation}
\sum_{k=j+1}^{\infty}d_{k}p_{k,j}=\alpha_j.
\end{equation}

On the other hand, if $L<\infty$, we obtain the following equality
\begin{displaymath}
  A_L=d_L.
\end{displaymath}
This finishes the proof.
\end{proof}

\begin{defi} \label{optimal:seq}
Let  $(p_{n,j})$ be as in Lemma \ref{lem:general} and let  $P^\alpha=(p_{n,j}^\alpha)$ be defined by
\begin{displaymath}
  p_{n,j}^\alpha=\left\{
    \begin{array}{ll}
      p_{n,j} & \textrm{ if } n\geq j_0+1, j\geq j_0\\
      0 & \textrm{ if } n\geq j_0+1, j< j_0\\
      1/n & \textrm{ if } n\leq j_0, j< j_0\\
    \end{array}
\right.
\end{displaymath}
\end{defi}
 
It is a direct consequence of Lemma \ref{lem:general} that  $P^\alpha\in\pi(\alpha)$.  
That is, $J_A(P^\alpha)\subset J_A(\alpha)$.

\subsection*{  Construction of the measure}
 Recall that, except for a countable number of points,  $x \in [0,1]$ can be written in a unique way in base $A$ as
\begin{equation*}
  x= \left[\epsilon_1(x) \epsilon_2(x) \dots \epsilon_n(x) \dots \right]_A.
\end{equation*}
Let $n\in \N$ and consider the cylinder set defined by
\begin{equation*}
  C(\epsilon_1, \dots, \epsilon_n) = \left\{ x \in [0,1] :  \epsilon_1(x)= \epsilon_1,
    \epsilon_2(x)= \epsilon_2, \dots,  \epsilon_n(x)= \epsilon_n    \right\}.
\end{equation*}
The collection of all cylinders form a semiring that generates the
Borel $\sigma-$algebra in $[0,1]$. Consider the probability measure
$\mu$ defined on cylinders by
\begin{equation}
  \label{eq:7}
  \mu  ( C(\epsilon_1, \dots, \epsilon_n)) 
  = \prod_{k=1}^{n} p_{b_k, \epsilon_k}^\alpha  = \prod_{k=1}^{n} r_{b_k}t_{\epsilon_k}.
\end{equation}
Note that,
\begin{equation*}
  \mu  \left( C(\epsilon_1(x), \dots, \epsilon_n(x)) \right) =
  \left( \prod_{j=0}^{\infty} 
    t_j ^{\tau_j(x,n)} \right) \left(  \prod_{i=2}^{\infty} r_i^{D_i(n)} \right).
\end{equation*}
Therefore, for every $x \in J_A(\alpha)$ we have that
\begin{equation}
  \label{eq:6}
  \begin{aligned}
    d_{\mu}(x)&= \lim_{r \to 0} \frac{\log \mu(B(x,r))}{\log r}=
    \lim_{n \to \infty} \frac{\log
      \mu(C(\epsilon_1(x),\ldots,\epsilon_n(x)))}{\log \prod_{i=1}^m b_i}\\
    &=
    \frac{ \lim_{n\to\infty}\frac{1}{n}\log
      \mu(C(\epsilon_1(x),\ldots,\epsilon_n(x)))}
    { \sum_{m=2}^{\infty} d_{m} \log m}\\
    &=\frac{\lim_{n\to\infty}\frac{1}{n}\log \left(
        \prod_{j=0}^{n-1} t_j^{\tau_j(n)}
        \prod_{i=2}^nr_i^{D_i(n)} \right)}{ \sum_{m=2}^{\infty}
      d_{m} \log m}
    \\
    &=\frac{\lim_{n\to\infty} \left( \sum_{j=0}^{n-1}
        \frac{\tau_j(n)}{n} \log t_j + \sum_{i=2}^n \frac{D_i(n)}{n}
        \log r_i \right)} { \sum_{m=2}^{\infty} d_{m} \log m}
    \\
    &=\frac{ \sum_{j=j_0}^{\infty} \alpha_j \log t_j +
      \sum_{i=j_0+1}^{\infty} d_{i} \log r_i} { \sum_{m=2}^{\infty}
      d_{m} \log m}.
  \end{aligned}
 \end{equation}
In virtue of Proposition \ref{ub} we obtain that
\begin{equation}
  \dim_H(J_A(\alpha)) \leq \frac{ \sum_{j=j_0}^{\infty} \alpha_j \log t_j +
    \sum_{i=j_0+1}^{\infty} d_{i} \log r_i} { \sum_{m=2}^{\infty}
    d_{m} \log m}.
\end{equation}
In order to obtain the lower bound note that since $J_A(P^\alpha)\subset J_A(\alpha)$
we have that $\dim_H(J_A(P^\alpha)) \leq \dim_H(J_A(\alpha))$.
Proceeding as in (\ref{eq:6}) and making use of  (\ref{eq:7}) we obtain that for every 
$x\in J_A(P^\alpha)$, the pointwise dimension is given by
\begin{displaymath}
  d_{\mu}(x)= \frac{\sum_{n=1}^{\infty} d_{n}
    \sum_{i=0}^{n-1} p_{i,n}^\alpha \log p_{i,n}^\alpha} {\sum_{n=2}^{\infty} d_{n}
    \log n}.
\end{displaymath}

Since $J_A(P^\alpha)\subset J_A(\alpha)$, we have that
\begin{displaymath}
  \frac{\sum_{n=1}^{\infty} d_{n}
    \sum_{i=0}^{n-1} p_{i,n}^\alpha \log p_{i,n}^\alpha} {\sum_{n=2}^{\infty} d_{n}
    \log n}=\frac{ \sum_{j=j_0}^{\infty} \alpha_j \log t_j +
    \sum_{i=j_0+1}^{\infty} d_{i} \log r_i} { \sum_{m=2}^{\infty}
    d_{m} \log m}
\end{displaymath}
and then, by \eqref{eq:3}, we obtain
\begin{displaymath}
  \dim_H(J_A(\alpha)) \geq \frac{ \sum_{j=j_0}^{\infty} \alpha_j \log t_j +
    \sum_{i=j_0+1}^{\infty} d_{i} \log r_i} { \sum_{m=2}^{\infty}
    d_{m} \log m}.
\end{displaymath}
This finishes the proof of the theorem.

 \end{document}